\documentclass[11pt]{amsart}
\usepackage{hyperref}
\usepackage{amsmath,amssymb,latexsym,amsthm}
\usepackage{url}
\usepackage{graphicx,color}
\newtheorem{thm}{Theorem}

\newtheorem{example}[thm]{Example}
\newtheorem{coro}[thm]{Corollary}

\date\today

\DeclareMathOperator{\supp}{supp}

\DeclareMathOperator{\WF}{WF} 
\DeclareMathOperator{\dist}{dist} 
 
\newcommand{\eps}{\varepsilon}
 
\newcommand{\R}{\mathbb{R}} 
\newcommand{\Id}{\mbox{Id}} 
\renewcommand{\r}[1]{(\ref{#1})} 
\newcommand{\PDO}{$\Psi$DO} 
\newcommand{\be}[1]{\begin{equation}\label{#1}} 
\newcommand{\ee}{\end{equation}} 

\renewcommand{\d}{\mathrm{d}}

\renewcommand{\i}{\mathrm{i}} 
 
\newcommand{\bo}{{\partial \Omega}}

\title[Thermo-acoustic Tomography with planar detectors]{Thermo-acoustic Tomography with planar detectors}
\author[P. Stefanov]{Plamen Stefanov} 
\address{Department of Mathematics, Purdue University, West Lafayette, IN 47907} 
\thanks{First author partly supported by a NSF  Grant DMS-1301646} 

\title[Thermo and photoacoustic Tomography with  planar detectors]{Thermo and photoacoustic Tomography with variable speed and planar detectors}
\author[Yang Yang]{Yang Yang} 
\address{Department of Mathematics, Purdue University, West Lafayette, IN 47907} 

\begin{document}
\begin{abstract}
We analyze the mathematical model of multiwave tomography with a variable speed with integrating  measurements on planes tangent to a sphere surrounding the source. We prove sharp uniqueness and stability estimates with full and partial data and propose a time reversal algorithm which recovers the visible singularities. 
\end{abstract}
\maketitle

\section{Introduction}

In multiwave tomography, a certain excitation is send to the medium which creates a source of ultrasound signal measure outside the patient. The most popular modalities are thermoacoustic tomography, where a microwave illumination is used to create the ultrasound; and photoacoustic tomography, where one excites the medium with laser light. 
The ultrasound pressure is modeled by the acoustic wave equation
\begin{equation} \label{IVP}
\left\{
\begin{array}{rcl}
(\partial^2_t - c^2(x)\Delta) u &=& 0 \quad\quad  \text{ in } (0,T)\times\mathbb{R}^n, \\
u|_{t=0} &=& f, \\
\partial_t u|_{t=0} &=& 0,
\end{array}
\right.
\end{equation}
where $T>0$ is fixed. Here $\Omega$ is a bounded open subset of $\mathbb{R}^n$ with smooth boundary $\partial\Omega$ and  $f$ is a function supported in $\overline{\Omega}$. Without loss of generality we may assume $\overline{\Omega}\subset B(0,1)$ where $B(0,1)$ denotes the open unit ball in $\mathbb{R}^n$ whose boundary is the unit sphere $\mathbb{S}^{n-1}$. 
The acoustic speed $c(x)>0$ is a smooth function in $\mathbb{R}^n$ with $c\equiv 1$ outside of $\Omega$. The results extend to general second order operators involving a metric, a magnetic and an electric field as in \cite{SU2009}. 
The inverse source problem in multiwave tomography is to recover the initial data $f(x)$ from the measurement of the acoustic waves. The measurement in the conventional model is pointwise, namely one assumes accessibility to $u|_{[0,T]\times \Gamma}$ where $u$ is the solution of \eqref{IVP} and $\Gamma$ is a relatively open subset of the boundary $\partial\Omega$. When $\Gamma=\partial\Omega$ the wave is measured on the full boundary; when $\Gamma\subsetneq\partial\Omega$ it is measured on partial boundary.  
 The mathematical model with pointwise measurements has been studied extensively, see, e.g.,   \cite{KuchmentK_11, SU2009,S-U-InsideOut10} and the references there.

For  pointwise measurements, the size of the transducers limits the resolution of the image reconstruction. Researchers have  designed alternative acquisition schemes using receivers of different shapes such as planar detectors \cite{BHPHS2005, HBPS2004}, and linear and circular detectors \cite{BHMPHS2006, GHPB2007, PNHB2007, ZSH2009}. They are also called integrating detectors since the signal  is integrated over the detector: each measurement returns a number and the detectors are rotated around the object, collecting more measurements. 
In this paper, we  consider the measurement made by planar detectors tangent to a sphere surrounding the object. When $c$ is constant, this type of measurement is studied in \cite{BHPHS2005, HBPS2004} and the problem reduces to the inversion of the Radon transform with limited data, see Theorem~\ref{thm_unique} below. We are interested in   variable sound speeds $c(x)$.

To define the measurement we recall the definition of the well known Radon transform: given a function $g(x)$ in $\mathbb{R}^n$, its Radon transform $Rg$ is a function of $(p,\omega)\in\mathbb{R}\times\mathbb{S}^{n-1}$ defined as
$$Rg(p,\omega):=\int_{x\cdot\omega=p} g(x) \,\d S(x)$$
where the integral is over the hyperplane $\{x\in\mathbb{R}^n: x\cdot\omega=p\}$ and $dS$ is the Lebesgue measurement on this hyperplane.  
Let $u(t,x)$ be the solution of \eqref{IVP} and $\Gamma$ a relatively open subset of $\mathbb{S}^{n-1}$. One way to define the planar measurement is  as the operator
\begin{equation} \label{plane}
Mf(t,\omega):=(Ru(t,\cdot))(1,\omega)=\int_{x\cdot\omega=1} u(t,x) \, \d S(x), \quad\quad\quad (t,\omega)\in (0,T)\times\Gamma.
\end{equation}
This corresponds to the measurement of the acoustic waves on the hyperplanes $\pi_\omega:=\{x: x\cdot\omega=1, \; \omega\in\Gamma\}$ tangent to the unit sphere over the time interval $(0,T)$.

The measurement operator $M$ assumes that the waves propagate through the measurements plane. It leads to an interesting mathematical problem but we also define a measurement  operator $N$ below by allowing reflections off the measuring plane,   imposing Neumann boundary conditions on it. If we assume that no geodesic starting from a plane $\pi_\omega$ perpendicularly comes back  to $\pi_\omega$ again perpendicularly (see assumption (\textbf{H}) below), then microlocally the problem is the same, as we show below.

  The operator $N$ allowing  $\pi_\omega$ to reflect the signal is defined as follows.  
The direct problem then changes with the measurements. 
Given $\omega\in \mathbb{S}^{n-1}$, we solve
\begin{equation} \label{IVP2}
\left\{
\begin{array}{rcl}
(\partial^2_t - c^2(x)\Delta) u &=& 0 \quad\quad  \text{for $0\le t\le T$, $x\cdot\omega\le 1$ }\\
\omega\cdot\nabla_x u |_{x\cdot\omega=1}&=&0,\\
u|_{t=0} &=& f, \\
\partial_t u|_{t=0} &=& 0,
\end{array}
\right.
\end{equation}
with $f$ supported in $\bar\Omega$ as above. We call the corresponding solution $u(t,x,\omega)$. Then we model the planar measurements by 
\begin{equation} \label{plane2}
Nf(t,\omega):=
\int_{x\cdot\omega=1} u(t,x,\omega) \, \d S(x).
\end{equation}
In this case, $Nf$ is the averaged Dirichlet data for this Neumann boundary value problem. 

Our main results are the following. We prove sharp uniqueness theorems with full and partial data in Theorem~\ref{thm_unique} and Theorem~\ref{thm_unique2} under the same conditions. We show that $N$ is microlocally equivalent to $2M$ in Theorem~\ref{thm_MN}. We characterize the measurements $M$ and therefore $N$ as Fourier Integral Operators (FIOs) in Theorem~\ref{FIO}. We give sharp conditions for stability with full and partial data and prove stability estimates in Theorems~\ref{thm_stability}, \ref{thm_stability_partial} and Theorem~\ref{thm_stab2}. In Corollary~\ref{cor9}, we characterize the visible singularities when there might be no stability. In section~\ref{sec_TR} we propose a time reversal algorithm that recovers the visible singularities of $f$; and in particular it recovers $f$ up to a smoothing operator, when there is stability. We use microlocal methods, and in particular, the calculus of FIOs, see, e.g., \cite{T1980_1,Hormander4}.

We would like to emphasize that even if one is interested in the measurements $N$ only (reflections), we need to analyze $M$ first both in the uniqueness theorems and in the stability ones, as well. Then $M$ can be considered as an auxiliary operator which analysis helps that of $N$. 

Finally, one could assume that the planes over which we take measurements are those tangent to a strictly convex closed surface instead of the unit sphere, and those methods would still work. Other types of boundary conditions in \r{IVP2} are possible, as well. 

\textbf{Acknowledgments.} The authors thank Guillaume Bal who attracted their attention to this problem.

\section{Preliminaries} 

We introduce some function spaces for the discussion below. Denote by $U$ an open domain of $\mathbb{R}^n$ which can be bounded or the whole $\mathbb{R}^n$. Let $\d x^2$ be the standard Euclidean measure, we will consider the conformal metric $c^{-2}\,\d x^2$ and the space $L^2(U):=L^2(U; c^{-2}\d x)$ consisting of square-integrable functions with respect to the measure $c^{-2} \d x$. Notice that the operator $c^2\Delta$ is formally self-adjoint with respect to the measure $c^{-2} \d x$. 
Define the Dirichlet space $H_{D}(U)$ to be the completion of $C^{\infty}_{0}(U)$ under the Dirichlet norm
$$\|f\|^2_{H_D(U)}:=\int_U |\nabla u|^2 \, \d x.$$
Here we actually integrate $c^2|\nabla u|^2$ with respect to $\frac{1}{c^2}dx$. When $U=\Omega$, it is easy to see that $H_{D}(\Omega)\subset H^1(\Omega)$ and that $H_{D}(\Omega)$ is topologically equivalent to $H^1_0(\Omega)$.

For a function $u=u(t,x)$, its energy is defined as
$$E_U(t,u):=\int_{U} (|\nabla u|^2 + \frac{1}{c^2}|u_t|^2) \, \d x.$$
Given Cauchy data $(f,\psi)$, we define the energy space $\mathcal{H}(U)$ by the norm
$$\|(f,\psi)\|^2_{\mathcal{H}(U)}:=\int_U (|\nabla f|^2 + \frac{1}{c^2}|\psi|^2) \,\d x.$$
The energy space admits the decomposition 
$$\mathcal{H}(U)=H_D(U)\oplus L^2(U)$$
and notice that
$$\|f\|^2_{H_D(U)}=(\Delta f, f)_{L^2(U)}.$$
The wave equation can be written as a system for $\mathbf{u_t}=(u,u_t)\in \mathcal{H}(U)$:  
$$\mathbf{u}=\mathbf{Pu}, \quad\quad \mathbf{P}=
\left(\begin{array}{cc}
	0 & I \\
	\Delta & 0 
\end{array}\right).
$$
The operator $\mathbf{P}$ extends to a skew self-adjoint operator on $\mathcal{H}(U)$, which by Stone's theorem generates a group of unitary operators $U(t)=\exp(t\mathbf{P})$. This justifies the well-posedness of the forward problem \eqref{IVP}. In particular it indicates that a natural function space for the consideration of $f$ is $H_D(\Omega)$.

For the Neumann problem \r{IVP2}, by finite speed of propagation, for any finite interval $t\in(0,T)$, we may assume that we work in a large domain $D$ with a part of the boundary being a part of $\pi_\omega$. The energy spaces then is given by the same norm but now we take the completion of $C^\infty(D)$ (no compactness of the support in $D$ assumed). Then the first component $f$ of $(f,\psi)\in\mathcal{H}$ is  defined up to a constant only. On the other hand, the solutions with $(1,0)$ as Cauchy data is $u=1$. This allows us to define solutions for all Cauchy data in $H^1(D)\times L^2(D)$ in a unique way. An alternative way is to use spectral methods. 

We assume below that $f\in H_D(\Omega)$ and supported in $\bar\Omega$, unless we say otherwise. The proofs are easily extended to distributions, as well. 

\section{Uniqueness}

We consider the uniqueness of the determination of $f$ from the measurement $Mf$ or $Nf$ in this section. We formulate below  sharp uniqueness results with full or partial measurements. Let $\Gamma\subset \mathbb{S}^{n-1}$ be a relatively open subset as before, and suppose we are restricted to making planar measurements on the planes $x\cdot\omega=1$ for $\omega\in \Gamma$ only.  To obtain information at an interior point, by finite speed of propagation, one needs to have at least one signal (i.e., a unit speed curve  with respect to the metric $c^{-2} \d x^2$) from that point to be detected by one of the planes $\pi_\omega$, $\omega\in\Gamma$.  As we show below, this is in fact a sharp time. Set
\[
T_0(\Omega, \Gamma) = \sup_{x\in\Omega} \inf_{\omega\in\Gamma}\dist(x,\pi_\omega),
\]
where the distance is with respect to the metric $c^{-2}\d x^2$. If $\Gamma=\bo$, it is easy to see that 
\[
T_0(\Omega,\bo) = \sup_{x\in\Omega}\dist(x,\mathbb{S}^{n-1}),
\]
because then any curve starting at $x$ minimizing $\dist(x,\pi_\omega)$ will hit $\mathbb{S}^{n-1}$ first before reaching $\pi_\omega$, and then will reach the plane tangent to the sphere at that point.

 The sharpness of $T_0$ follows from the unique continuation result of Tataru \cite{T1995,Tataru99}, as can be seen in the proof below. Similar sharp uniqueness results under other settings can be found in \cite{SU2009, SU2011, SY2015}.

\begin{thm} \label{thm_unique}
If  $\supp f\subset \bar\Omega\subset {B(0,1)}$, then 
$Mf(t,\omega)$ known for $\omega\in\Gamma$ and $0\le t\le T$ determines $f$ uniquely in the domain of  influence 
\[
\Omega_\Gamma:= \{x;\; \exists \omega\in \Gamma\;  \text{such that}\; \dist(x,\pi_\omega)<T  \}
\]
and $f$ can be arbitrary in $\Omega\setminus \bar\Omega_\Gamma$. 

In particular, if $T> T_0(\Omega, \Gamma)$, then $f$ is determined uniquely. 
\end{thm}
\begin{proof}
Let $u$ be the solution of \eqref{IVP} and let $U(t,p, \omega):=(Ru(t,\cdot))(p,\omega)$ be the Radon transform of $u$ for a fixed $t$. Since $c=1$ near the planes $x\cdot\omega=p>1$, the function $U(t,p,\omega):=(Ru(t,\cdot))(p,\omega)$ solves
\be{ST1}
\left\{
\begin{array}{rcll}
(\partial^2_t - \partial^2_s) U&=&0, &   \quad p>1,\quad  t\ge0,\\
U|_{p=1}&= &Mf(t,\omega), &\quad   t\ge0,\\
U|_{t=0} &=& 0,&\quad   p\ge 1,\\ 
\partial_t U|_{t=0}& =&0, &\quad   p\ge 1,\\
\end{array}
\right.               
\ee
The solution to this problem for $ p\ge 1$, $t\ge0$ is  given explicitly by
\be{ST2}
U(t,p,\omega) =
\left\{
\begin{array}{ll}
 Mf(t+1-p,\omega),   & 0\le p-1<t ,\\
0,     &0\le t\le p-1.
\end{array}
\right.               
\ee
This shows us that for every  $\omega\in\Gamma$,  $Mf(t,\omega)|_{(0,T)}$ determines $U(t,p,\omega)$ for $t-T+1<p$, $p>1$, $t\ge0$ in an explicit way. Since the problem is linear, we may assume that $Mf=0$ in the given set, and then we want to show that $f=0$ in the domain of influence. 
The solution $u$ extends  in an even way to $t<0$ as a solution, and the same applies to $U$. So in particular, we get  $U=0$ for $|t|<T$, $p>1$, $\omega\in\Gamma$. When $c=1$, \r{ST1} is valid for all $p$, and this leads us to the known solution of solving the problem then: we get the Radon transform of $f$ directly; and then invert it. 

Now, for every $t\in (-T,T)$, $u(t,\cdot)$ is  supported in $B(0,1+t)$ and its Radon transform vanishes for $p>1$, $\omega\in\Gamma$. 
 By the local support theorem for the Radon transform, see \cite{BQ1}, we get $u(t,x)=0$ for $x\cdot\omega>1$ for every $\omega\in\Gamma$. Therefore, in timespace, $u$ vanishes in an one-sided neighborhood of the   hyperplane $x\cdot\omega=1$, $t\in\R$ intersected with $|t|<T$. 
The theorem now follows by unique continuation. Indeed,  vanishing Cauchy data near every line $x=x_0$, $t\in (-T,T)$ in that set implies $u=0$ in its the domain of influence  $|t|+\dist(x,x_0)<T$ by Tataru's unique continuation theorem \cite{T1995,Tataru99}, see also \cite{SU2011}. 
In particular, when $t=0$ we get $f(x)=0$ when $\dist(x,x_0)<T$ for some $x_0\in\pi_\omega$ and  $\omega\in\Gamma$. 
\end{proof}

We prove a similar  uniqueness theorem for the operator $N$ next.  
\begin{thm} \label{thm_unique2} 
The uniqueness  Theorem~\ref{thm_unique} remains true with $M$  replaced by $N$. 
\end{thm}

\begin{proof}
Notice first that we can use the method of reflections to solve the direct problem \r{IVP2} by reflecting the solution of \r{IVP} that we call $u_0$ in this proof, as long as the reflected part of $u_0$ does not intersect $\Omega$. Indeed, let $x_\omega$ be the image of $x$ reflected about the plane $\pi_\omega$. Then $u_1$ defined as $u_1(t,x,\omega): = u_0(t,x)+u_0(t,x_\omega)$ for $x\cdot\omega<1$ satisfies the Neumann boundary condition on $\pi_\omega$ and solves the wave equation if $\supp_x u_0(t,x_\omega)$ does not intersect $\bar\Omega$ where $c$ might not be equal to one. Therefore, under this condition,  $u_1=u$. On the other hand, then $Nf(t,\omega)=u(t,x,\omega)|_{\pi_\omega}= 2u_0(t,x)|_{\pi_\omega}=2Mf(t,\omega)$.

The difficulty in using unique continuation is that we need to apply it to $\omega$ in an open set but $u$ depends on $\omega$. For this reason, we will reduce the problem to unique continuation for $u_0$ which is $\omega$ independent. 

Fix $\omega_0\in \mathbb{S}^{n-1}$. We extend the solutions of the forward problem for $t<0$ in an even way as before.  Assume first that $Nf(t,\omega)=0$ for $t\le T$ and $\omega$ in some neighborhood of $\omega_0$. 
We will prove that  $f=0$ in the domain of influence $\dist(x,\pi_{\omega_0})<T$. 

There is $\rho>0$ so that $f(x)=0$ for $\dist(x,\pi_{\omega_0})<\rho$. 
For $\omega$ close to $\omega_0$, consider $u(t,x,\omega)$ for  $|t|\in [\rho_0,\rho_0+\delta]$ with $0<\delta<1-\max(x\cdot\omega;\; x\in\bo)$ fixed. Then  $u(t,x,\omega)$ can be obtained from $u_0$ by a reflection, if $\omega$ is close enough to $\omega_0$ (depending on $\delta$). 
Then we get $Mf(t,\omega)=0$ for such $t$ and $\omega$ as long as $|t|<T$. 
Therefore, by Theorem~\ref{thm_unique},  $f(x)=0$ for $\dist(x,\pi_{\omega_0})<\rho+\delta$ if  $\rho+\delta<T$. Thus the supremum of such $\rho$ must be $T$.  We can vary $\omega_0$ over $\Gamma$ now to conclude the proof.
\end{proof}

\section{Stability}

In order to have a stable determination, one needs be able to detect all the microlocal singularities of $f$. By the propagation of singularity theory, every microlocal singularity $(x,\xi)\in T^\ast\Omega\backslash 0$ of $f$ splits into two singularities which then travel along the bi-characteristic curves $(\gamma_{x,\pm\hat\xi}(t), \dot{\gamma}_{x,\pm\hat \xi}(t))$, where $\hat\xi=\xi/(c|\xi|)$ is the unit covector in the direction of $\xi$. If we identify vectors and covectors by the metric $c^{-2}\d x^2$, then the bi-characteristic curves are the unit speed geodesics in $T\Omega$ issued form $(x,\hat\xi)$. These curves will eventually leave $\Omega$ if we assume that $c^{-2}\d x^2$  
is non-trapping. The latter means that all geodesics through $\bar\Omega$ are of finite length, and we assume it from now on. We show below that a singularity can be detected if and only if $\gamma_{x,\pm\xi}(t)$ hits some of the planes $\pi_\omega$ perpendicularly. 
There are exactly two values of $t$, say $t_\pm=t_{\pm}(x,\xi)$, such that $\gamma_{x,\pm\xi}(t)$ hits a tangent plane of $\mathbb{S}^{n-1}$ perpendicularly at $t=t_\pm$. Define
$$
T_1:=\frac{1}{2}\sup_{(x,\xi)\in S^\ast\Omega\backslash 0}|t_+(x,\xi)-t_-(x,\xi)|.
$$
We show below that this is the sharp time for the stability. Notice that the non-trapping assumption on $c$ is equivalent to  $T_1<\infty$.

\subsection{Stability analysis for $M$} 

We show that $M$ is a Fourier integral operator (FIO) and calculate its canonical relation. 
We will present first some heuristic arguments first which can be used as a basis for an alternative proof but that would require some geometric assumptions which are not needed for our results below. 
 The singularities of the kernel $M(t, \omega,y)$ of $M$ can be described in the following way. For $y$ fixed, the solution $u(t,x)$ corresponding to $f= \delta_y(x)$ has singular support on the geodesic sphere $\dist(x,y)=t$, where $\dist$ is the distance in the metric. Those spheres would be smooth only if (i) $y$ does not have conjugate points.   The wave front set would be conormal to it. Now, integrating over the plane $\pi_\omega$ for $t>0$ fixed would create a singularity only if that plane is tangent to the geodesic sphere (when the latter is smooth). Therefore, $M$ is singular on the manifold
\[
Z := \{ (t,\omega,y);\; \dist(x,\pi_\omega)=t \}, 
\]
 when (ii)  there is a unique minimizing geodesic realizing that distance. This, in particular implies that $\omega$ is equal to the unit tangent to that geodesic at the intersection point with $\pi_\omega$, and that the geodesic hits $\pi_\omega$ perpendicularly, see Figure~\ref{fig}.   Then $M$ must be an FIO with a Lagrangian $N^*Z$. One can use this to prove the results below   under the  assumptions (i), (ii) above, and to get the visibility condition below. This description resembles the double fibration formalism in integral geometry. In particular, we see (under the assumptions that we remove below) that a singularity $(x,\xi)$ can only be detected by $Mf$ near some $(t,\omega)$ if  $\gamma_{x,\hat\xi}$  hits the plane $\pi_\omega$ perpendicularly at time $t$ or $-t$. As we see below, (i) and (ii) are not needed, and in general, the Lagrangian associated with $M$ and $M$ is not of conormal type $N^*Z$.

\begin{figure}[h!] 
  \centering
  \includegraphics[trim = 0mm 0mm 0mm 0mm, clip, scale=0.6]{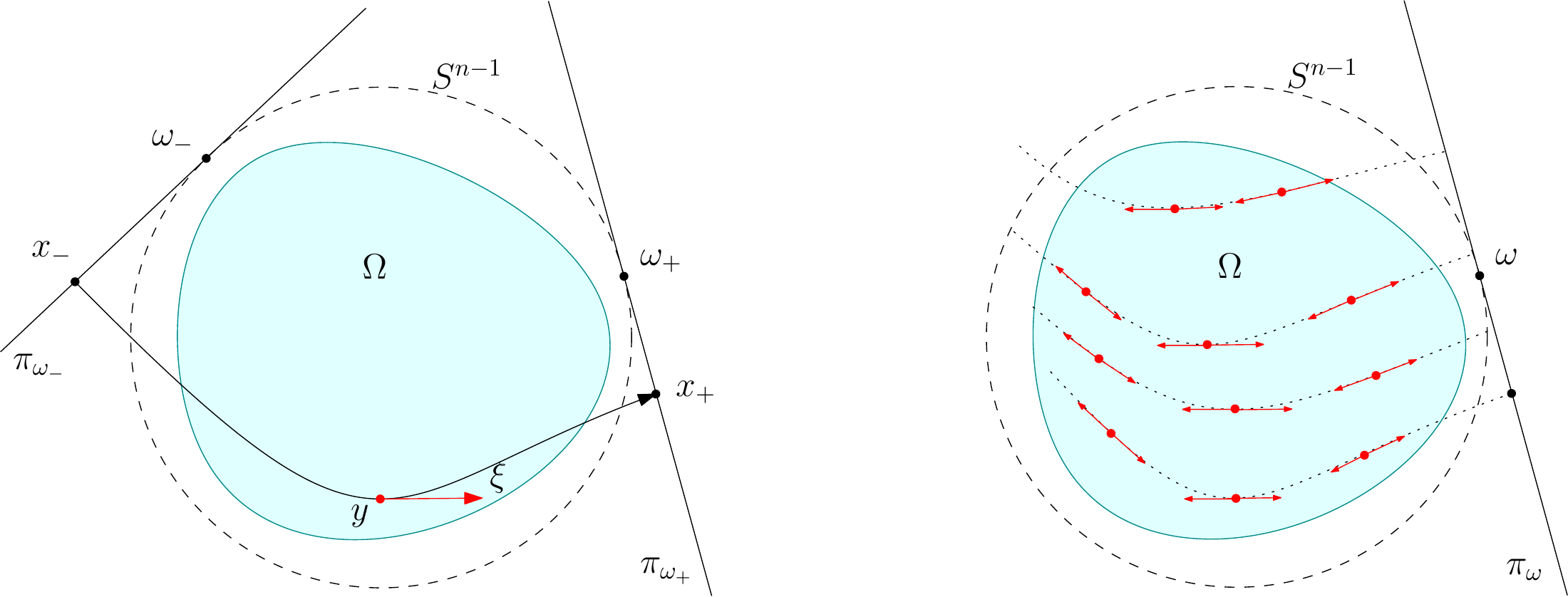}
\caption{\small Left:  The canonical relations $C_\pm: (y,\xi)\to (t_\pm,\omega_\pm, *,*)$, where the dual variables are denoted by $*$. The unit speed geodesic through $(x,\xi)$ hits $\pi_{\omega_\pm}$ perpendicularly at $x_\pm$ at time $\pm t_\pm>0$. Thus the singularity $(x,\xi)$ can be detected by measurements on $\pi_{\omega_\pm}$ at time $t=t_\pm$.  The points $x_\pm$ are determined by the dual variables $(*,*)$. 
Right: Singularities visible from a neighborhood of a single $\pi_\omega$, $T\gg1$. 
}
\label{fig}
\end{figure}

We begin by constructing a parametrix to the problem \eqref{IVP}, see also \cite{SU2011}.  Fix $x_0\in\Omega$, in a neighborhood of $(t,x)=(0,x_0)$ the solution of \eqref{IVP} is given by
\begin{equation} \label{GOS}
u(t,x)=(2\pi)^{-n} \sum_{\sigma=\pm} \int e^{i\phi_\sigma(t,x,\xi)} a_\sigma(t,x,\xi) \hat{f}(\xi) \,\d\xi
\end{equation}
modulo smooth terms. Here the phase functions $\phi_\pm$ are positively homogeneous of order $1$ in $\xi$ and solve the eikonal equations
$$\pm \partial_t \phi_\pm + c(x)|d_x\phi_\pm|=0, \quad\quad \phi_\pm|_{t=0}=x\cdot\xi$$
where $|\cdot|$ is the Euclidean norm. The amplitudes $a_\pm$ are classical of order $0$ and solve the corresponding transport equations with initial conditions $a_\pm(0, x,\xi)=\frac{1}{2}$ \cite[eqn. VI.1.50]{T1980_1}. In particular, in the asymptotic expansion $a_\pm\sim\sum_{j\geq 0} a^{(-j)}_\pm$ with $a^{(-j)}_\pm$ homogeneous in $\xi$ of order $-j$, the leading terms $a^{(0)}_\pm$ satisfies the following homogeneous transport equation and initial conditions
\begin{equation}\label{8}
\left(\partial_t - c^2(x) (\partial_j\phi_\pm) \partial_j + C_\pm \right) a^{(0)}_\pm = 0, \quad\quad a^{(0)}_\pm|_{t=0}=\frac{1}{2},
\end{equation}
where $C_\pm$ are smooth multiplication terms.

To obtain an oscillatory integral representation of the operator $M$, we apply the Radon transform to \eqref{GOS} at $p=1$. We consider only the term with the $+$ sign in \eqref{GOS} and write $\phi:=\phi_+$ and $a:=a_+$ for simplicity of notations. The analysis of the ``$-$'' term is similar. The construction \eqref{GOS} is valid as long as the eikonal equation is solvable. This is always true locally. We assume  that the solution, microlocalized for $f$ with $\WF(f)$ near some $(x_0,\xi^0)$ extends all the way until the geodesics $\gamma_{x_0,\xi^0}$ hits a plane $\pi_\omega$, and even in some neighborhood of that interval. This condition can easily removed as in \cite{SU2009}. Then
\begin{align}
Mf(t,\omega)& =  (2\pi)^{-n} \int_{x\cdot\omega=1} \int e^{i\phi(t,x,\xi)}a(t,x,\xi) \hat{f}(\xi) \,\d\xi \,\d S(x) \nonumber \\
 &= (2\pi)^{-n} \iint e^{i\phi(t,x,\xi)}a(t,x,\xi) \hat{f}(\xi) \delta(1-x\cdot\omega) \,\d\xi \,\d x \nonumber \\
 &= (2\pi)^{-n-1} \iiint e^{i\phi(t,x,\xi)+i\lambda(1-x\cdot\omega)}a(t,x,\xi) \hat{f}(\xi) \,\d\xi\, \d x \,\d\lambda.  \label{Mp}
\end{align}

Write $\hat f(\xi) = \int e^{-i y\cdot\xi} f(y)\d y$. Then the phase function becomes 
$$\Phi(t,\omega, y; x,\lambda, \xi):=\phi(t,x,\xi)+\lambda(1-x\cdot\omega)-y\cdot\xi.$$ 
Here the spatial variables are $(t,\omega, y)$ and the fiber variables are $\theta:= (x,\lambda, \xi)$. The issue with $\Phi$ is that it is not homogeneous of degree $1$ with respect to $x$. This can be resolved by introducing $\tilde{x}:=x|(\xi,\lambda)|$ with $|(\xi,\lambda)|:=(|\xi|^2+|\lambda|^2)^{\frac{1}{2}}$ and defining a new phase function (see \cite[Proposition 21.2.19]{H1985})
$$\tilde{\Phi}(t,\omega, y; \tilde{x},\lambda, \xi):=\Phi\Big(t,\omega, y; \frac{\tilde{x}}{|(\xi,\lambda)|}, \lambda, \xi\Big).$$
It is easy to see that when $(x, \lambda, \xi)\neq 0$, $\tilde{\Phi}$ is smooth, homogeneous of degree $1$ in the fiber variables, and $\tilde{\Phi}_{(t,\omega,x,\lambda, \xi)}$ and $\tilde{\Phi}_{(y,x,\lambda, \xi)}$ are non-vanishing, thus $\tilde{\Phi}$ is a phase function in the sense of \cite[VI.2]{T1980_1}.

Making a change of variable $x\mapsto \tilde{x}$ in \eqref{Mp} one obtains
$$Mf(t,\omega)=(2\pi)^{-n-1}\iiint e^{i\tilde{\Phi}(t,\omega,y; \tilde{x}, \lambda, \xi)}\tilde{a}(t,\tilde{x},\lambda,\xi) \hat{f}(\xi) \, \d \xi \d \tilde{x} \d \lambda$$
where $\tilde{a}(t,\tilde{x},\lambda, \xi):=a(t,\frac{\tilde{x}}{|(\xi,\lambda)|},\xi)|(\xi,\lambda)|^{-n}$ is the new amplitude. This indicates that $M$ is an elliptic FIO of order $\frac{1-n}{2}$ \cite[Definition 3.2.2]{H1971}.

Next we compute the canonical relation of $M$ and show that it is a local graph. Since by the chain rule 
$$\tilde{\Phi}_\xi=\Phi_\xi+\Phi_x \left(\frac{\tilde{x}}{|(\xi,\lambda)|}\right)_\xi, \quad\quad \tilde{\Phi}_{\tilde{x}}=\Phi_x \frac{1}{|(\xi,\lambda)|}, \quad\quad \tilde{\Phi}_\lambda=\Phi_\lambda + \Phi_x \left(\frac{\tilde{x}}{|(\xi,\lambda)|}\right)_\lambda,$$
the replacement of $\Phi$ by $\tilde{\Phi}$ does not affect the  characteristic manifold $\Sigma := \{\Phi_\theta=0\}$:
\begin{align*}
\Sigma & =\{(t,\omega, y; x, \lambda, \xi): \Phi_\xi=0, \Phi_x=0, \Phi_\lambda=0\} \\
 & =\{(t,\omega, y; x, \lambda, \xi): y=\phi_\xi, \phi_x=\lambda\omega, x\cdot\omega=1\}.
\end{align*}
By the geometric optics construction, see, e.g.,  \cite[VI.2 Example 2.1]{T1980_1}, one sees that $y=\phi_\xi$ implies that $x$ is on the geodesic $\gamma_{y,\hat\xi}$ issued from $(y,\hat\xi)$, where $\hat\xi = \xi/(c|\xi| )$ is the unit covector in the metric identified with a unit vector,  and $(\gamma_{y,\hat\xi}(t),c|\xi| \dot{\gamma}_{y,\hat\xi}(t))=(x,\phi_x)$. The condition $x\cdot\omega=1$ means $x$ is the intersection of the geodesic $\gamma_{y,\xi}$ and the plane $x\cdot\omega=1$, as a result $t=t_+(y,\xi)$ is the time of the intersection. The condition $\phi_x=\lambda\omega$ says the tangent vector $\dot{\gamma}_{y,\xi}(t)$ is in the direction of  $\omega$, i.e., the geodesic $\gamma_{y,\xi}$ hits the plane $x\cdot\omega=1$ perpendicularly, see Figure~\ref{fig}.  
As the intersection occurs outside of $B(0,1)$ and $c=1$ there, one sees that $\lambda=c(y)|\xi| |\dot{\gamma}_{y,\hat\xi}(t)| = c(y)|\xi|$ and $\omega=\widehat{\dot{\gamma}}_{y,\hat{\xi}}(t)$ where $|\cdot|$ is the Euclidean norm. 
If we denote the time that $\gamma_{y,\hat{\xi}}$ hits $\partial B(0,1)$ by $t_0=t_0(y,\xi)$, then we also have 
 $\omega=\widehat{\dot{\gamma}}_{y,\hat{\xi}}(t_0)$ since $\gamma_{y,\hat{\xi}}$ is a straight line outside of $B(0,1)$. 
This argument shows that $\Sigma$ is a smooth manifold parameterized by $(y,\xi)$ and hence of dimension $2n$. 

We include the phase function $\phi_-$ now, as well, and call the corresponding characteristic variety $\Sigma_-$.  Then the corresponding time of intersection with the plane $\pi_\omega$ is $t=t_-(x,\xi)<0$. Also, $\dot\gamma$ at this time points in the opposite direction of $\omega$, therefore, $\lambda$ changes sign. 
Therefore,  each of the maps  (for $\Phi$ associated with $\phi_\pm$)
$$\Sigma_\pm \ni (t,\omega, y; x,\lambda, \xi) \longmapsto 
 (t,\omega,y; \Phi_t, \Phi_\omega,\Phi_y)
=(t,\omega,y; \mp c(y)|\xi|, \mp c(y)|\xi|(x - \omega), -\xi)$$
is smooth of rank $2n$ at any point, thus $\Phi$ is a non-degenerate phase \cite[VIII.1]{T1980_1} and the canonical relation is
a local graph given by
\begin{align*}
C: = & \{( t,\omega, \mp c(y)|\xi|, \mp c(y)|\xi|(x - \omega); y,\xi), \quad (t,\omega, y; x,\lambda, \xi)\in\Sigma\} \\
 = & \left\{( t_\pm(y,\xi), \pm\widehat{\dot{\gamma}}_{y,\xi}(t(y,\xi)), \mp c(y)|\xi|,\mp c(y)|\xi|\Big(\gamma_{y,\xi}(t(y,\xi)) - \widehat{\dot{\gamma}}_{y,\xi}(t(y,\xi))\Big); y,\xi), \right. \\
 &  \quad\quad\quad\quad\quad \left. (y,\xi)\in T^\ast\Omega\backslash 0\right\}
\end{align*}
Note that $x-\omega$ is the projection of $x\in \pi_\omega$ on the tangent space $T_\omega\mathbb{S}^{n-1}$, which is also the derivative of $x\cdot\omega$ with respect to $\omega\in\mathbb{S}^{n-1}$. 

Putting the above analysis together, we showed

\begin{thm} \label{FIO}
The operator $M= M_+ + M_-$, where $M_\pm$ are elliptic Fourier integral operators of order $\frac{1-n}{2}$ with canonical relations given by the graphs of the maps
$$
C_\pm: (y,\xi) \longmapsto \left(t, \pm\widehat{\dot{\gamma}}_{y,\xi}(t), \mp c(y)|\xi|, -c(y)|\xi|\big(\gamma_{y,\xi}(t) - \widehat{\dot{\gamma}}_{y,\xi}(t)\big) \right), \quad t=  t_\pm(y,\xi). 
$$
\end{thm}
The canonical relations above are of the form $(y,\xi) \mapsto ( t,\omega,\tau,\omega^\sharp )$, where $(\tau,\omega^\sharp)$ are duals to $(t,\omega)$.

\textbf{Remark:} Another way to see that $M$ is a Fourier integral operator is to regard it as the composition of the solution operator of the wave equation and the  Radon transform w.r.t.\ $x$ at $p=1$.

The stability of the determination follows from the above theorem. We introduce a cut-off function $\chi\in C^\infty_0(0,T)$ so that 
$\chi>0$ on $[0,T_1]$
  and model the finite time measurement with $\chi Mf$. This way, we can simply define the fractional Sobolev norm of $Mf$ below by extending $Mf$  as zero for all $t$.

\begin{thm} \label{thm_stability}
Suppose $\supp f\subset\bar\Omega\subset  {B(0,1)}$ and $T>T_1$. Then we have the stability estimate
$$\|f\|_{H^1(\Omega)} \leq C \|\chi Mf\|_{H^{\frac{1+n}{2}}((0,T)\times\mathbb{S}^{n-1})}$$
for some constant $C>0$ independent of $f$.
\end{thm}

\begin{proof}
Since $M$ is an elliptic FIO of order $\frac{1-n}{2}$ associated to the canonical graphs $C_\pm$, its adjoint $M^\ast$ is also an elliptic FIO of the same order associated to the canonical graphs $C^{-1}_\pm$.  
Thus $M^\ast \chi M$ is an elliptic pseudodifferential operator of order $1-n$ in a neighborhood of $\Omega$ with a positive homogeneous principal symbol on the unit cotangent bundle. It follows from the elliptic regularity estimate and the mapping property of $M^\ast$ that
$$\|f\|_{H^1(\Omega)} \leq C(\; \|\chi M f\|_{H^{\frac{1+n}{2}}((0,T)\times\mathbb{S}^{n-1})} + \|f\|_{L^2(\Omega)} \;).$$ 
Since Theorem \ref{thm_unique} implies that $\chi M$ is injective on $H_D(\Omega)$, by \cite[Proposition V.3.1]{T1981} we can get rid of the last term on the right and obtain the desired estimate, with possibly a different constant $C>0$.
\end{proof}

In the same way, one can prove an $L^2\to H^{(n-1)/2}$ estimate as well here; and also in the theorem below. Note that those estimates are in sharp norms, since $M$ is an FIO of order $(1-n)/2$ associated with a local canonical diffeomorphism.

Next we generalize the above theorem to the partial data case. Suppose $\Gamma\subset\mathbb{S}^{n-1}$ is as in Theorem \ref{thm_unique}, and suppose the function $f$ is always supported in some fixed compact set $K\subset\Omega$. In order to ensure the detection of all the singularities by the planes in $\Gamma$ we require
\begin{equation} \label{sta_cond}
\forall (y,\xi)\in S^\ast K, \; (t_\sigma(y,\xi), \dot{\gamma}_{y,\xi}(t_\sigma(y,\xi)))\in (0,T)\times\Gamma \text{ for at least one of } \sigma = + \text{ or } \sigma = -.
\end{equation}
Let $T_1(\Gamma,K)$ be 
the infimum of  $T$ for which \r{sta_cond} holds, and fix $T>T_1$. By compactness argument, \r{sta_cond} remains true if we replace $\Gamma$ with a compact subset $\Gamma_K$. Choose   $\chi\in C^\infty_0((0,T)\times\Gamma)$ so that $\chi>0$ on $[0,T_1]\times\Gamma_K$. We  model the partial measurement by  $\chi Mf$. Similar reasoning as above yields the following partial data stability result.

\begin{thm} \label{thm_stability_partial}
Suppose $K\subset\Omega$ is a fixed compact set and $T>T_1(\Gamma,K)$. Then we have the stability estimate for $f$ with $\supp f\subset K$:
$$\|f\|_{H_D(K)} \leq C \|\chi Mf\|_{H^{\frac{1+n}{2}}((0,T)\times\Gamma)}$$
for some constant $C>0$ independent of $f$.
\end{thm}

\begin{example}\rm 
An example of stable set $\Gamma$ is the following. Let $c=1$ and let $\Gamma$ be any open set on  $\mathbb{S}^{n-1}$ so that $\Gamma \cup (- \Gamma) = \mathbb{S}^{n-1}$. One choice is some neighborhood of a closed hemisphere. Then every lone through the unit ball intersects $\Gamma$, and $T>2$ with that $\Gamma$ implies stability. In this case ($c=1$) $Mf$ relates directly to the Radon transform, see the proof of Theorem~\ref{thm_unique}, therefore the stability condition reduces to well known properties of the Radon transform for $c$ constant.
\end{example}

\subsection{Stability analysis for the reflectors model} We analyze here the stability of recovery $f$ given $Nf$. We show that we can reduce the analysis to the one above. 

The method of reflections we used to prove uniqueness does not work anymore when the reflected wave intersects the region where $c$ is variable. Microlocally however, reflections work in the following sense. Singularity hitting $\pi_\omega$ is never tangent to it and it would reflect from it according to the laws of geometric optics. The leading amplitude in \r{GOS} will preserve its value and sign on the plane (and would alter the sign if we had Dirichlet boundary conditions). It may hit the same plane again at a later time. If it does not, the contribution of that reflected way to $Nf$ is a smoothing operator. On the other hand, then $Nf$ equals $2Mf$ up to a smoothing term, so we have essentially the same microlocal information as above. We make this more precise below.

As we mentioned above, it is convenient to make the following assumption:

\medskip 
(\textbf{H}) There is no geodesic in the metric $c^{-2}\d x^2$ of length $T$ with endpoints on some of the planes $\pi_\omega$, normal at it at both endpoints. 
\medskip 

This condition holds when $c$ is close enough to a constant, for example. It is not really necessary for the analysis since we can use the methods in \cite{SY2015} then. It makes the exposition simpler however. 

In this case, the geometric optics construction is well known. We start with   \r{GOS}, and extend it microlocally until the singularities hit $\pi_\omega$, and go a bit beyond it. Call this solution $u_0$. Then we find the boundary trace of $u_0$ on $\pi_\omega$ and construct a parametrix $u_R$ with that trace propagating into the future. We refer to \cite{SY2015}, for example, for more details. Then $u=u_0+u_R$ is the desired parametrix. Its singularities issued at normal directions never come back at normal directions again, by (H).  For its boundary values, we have $u|_{\pi_\omega}= 2u_0|_{\pi_\omega}$ and this is true for all $t$ by (H). This yields the following.

\begin{thm} \label{thm_MN}
$2M-N$ is a smoothing operator. 
\end{thm}

   The analysis above therefore yields the following. 

\begin{thm}\label{thm_stab2}
Under  assumption (H), Theorem~\ref{thm_stability} and Theorem~\ref{thm_stability_partial} remain true for the operator $N$, as well. 
\end{thm}

\subsection{Visible Singularities}

In this section we study which singularities, i.e., elements of the wave front set of $\WF(f)$ of $f$ can be recovered in stable way from $Mf$ or $Nf$. By that, we mean that they create singularities of $Mf$ or $Nf$, which in turn implies stability estimates in Sobolev spaces.  We  consider the functions $f$ supported in $\bar\Omega$, as before. 
 Since $M$ is an elliptic FIO associated with a local canonical diffeomorphism,  we obtain, see \cite{Hormander4},
$$ \WF{(Mf)} = C \circ \WF{(f)},$$
where $C=C_+\cap C_-$, see Theorem \ref{FIO}.

Let $U$ be a neighborhood of a fixed point $(t_0,\omega_0)$ in $\mathbb{R}\times\mathbb{S}^{n-1}$. A singularity $(y,\xi)\in\WF{(f)}$ is called \emph{visible} from $U$ if it creates a singularity in the limited data $Mf|_{U}$. Next we characterize all the singularities which are visible from $U$. Propagation of singularity theory shows that any $(y,\xi)\in\WF{(f)}$ splits into two singularities and they propagate along the bicharacteristic curves $(\gamma_{y,\hat\xi}(t), c|\xi|\dot{\gamma}_{y,\hat\xi}(t))$. Each singularity is later detected by a plane $\{x\cdot\omega=1\}$ which it hits perpendicularly at time $t$. Thus to trace back to the visible singularities in from a neighborhood of some $(t,\omega)\in U$, we can take all the geodesics issued from the plane $\pi_\omega$ in the direction $-\omega$ and extend them to time $t$, see also Figure~\ref{fig}.  Since $2M$ and $N$ are microlocally equivalent, we get the following. 
\begin{coro}\label{cor9}
The singularities of $f$ which are visible from $U$ for the measurements operators $M$ or $N$  are characterized by
$$
\WF{(f)} \cap \{(\gamma_{x,-\omega}(t),\lambda\dot{\gamma}_{x,-\omega}(t)) : x\cdot\omega=1, \; \lambda\in\mathbb{R}\backslash \{0\}, \;\; (t,\omega) \in U\}.
$$
\end{coro}

This corollary can be microlocalized: we can describe the singularities visible from an open conic subset of $T^*(\R\times \mathbb{S}^{n-1})$.  The corollary can also be derived from Theorem~\ref{thm_TR_p} below.

\section{Time Reversal} \label{sec_TR}
In this section, we propose time reversal algorithms which can be implemented numerically in an easy way and recover the visible singularities of $f$.

By the proof of Theorem~\ref{thm_unique}, for every fixed $T>0$, we can recover the Radon transform $[Ru(T,\cdot)](p,\omega)$ of $u(T,\cdot)$ for $p>1$ in an explicit way by $[Ru(T,\cdot)](p,\omega)=Mf(T+1-p,\omega)$, where $Mf(t,\omega)$ is extended as $0$ for $t<0$. We can differentiate this w.r.t.\ $p$, and then we see that we can recover the translation representation $[\mathcal{R}\mathbf{u}(T)] (p,\omega)$ for $p>1$. Recall that the Lax-Phillips translation representation  \cite{LP} of $\mathbf{f}=(f_1,f_2)$ is given by
\[
\mathcal{R}\mathbf{f}(p,\omega) = c_n\left(-\partial_p^{(n+1)/2}Rf_1    + \partial_p^{(n-1)/2}Rf_2    \right), \quad c_n:= \frac12 (2\pi)^{(1-n)/2}
\]
for $n\ge3$ odd, which we assume from now on. It is known that $\mathcal{R}:\mathcal{H}_0\to L^2(\R\times \mathbb{S}^{n-1})$ is unitary. The inverse is given by 
\be{TD_9}
\mathcal{R}^{-1}k(x) = 2c_n^-\int_{\mathbb{S}^{n-1}}\left(-\partial_s^{(n-3)/2} k(x\cdot\omega,\omega), \; \partial_s^{(n-1)/2}k(x\cdot\omega,\omega)\right)\d\omega, \quad c_n^-:= \frac12 (-2\pi)^{(1-n)/2}.
\ee
Then, for $p>1$,
\[
\begin{split}
[\mathcal{R}&\mathbf{u}(T)] (p,\omega) = \\
&=c_n\left(-\partial_p^{(n+1)/2}Mf(T+1-p,\omega)    + \partial_p^{(n-1)/2}\partial_t Mf(T+1-p,\omega)   \right)\\
&= c_n\left(-[(-\partial_t)^{(n+1)/2}Mf](T+1-p,\omega)    + [(-\partial_t)^{(n-1)/2}\partial_t Mf](T+1-p,\omega)   \right)\\
& = 2c_n^-[\partial_t^{(n+1)/2}Mf](T+1-p,\omega).
\end{split}
\]
If we knew $\mathcal{R}\mathbf{u}(T)] (p,\omega) $ for all $p$ (and $\omega$), we could invert $\mathcal{R}$, get $\mathbf{u}(T)=(u,u_t)|_{t=T}$, and solve the wave equation with speed $c$ from $t=T$ to $t=0$. One naive attempt to do time reversal in our case would be to extend $\mathcal{R}\mathbf{u}(T)$ as zero for $0\le p\le 1$ and then apply $\mathcal{R}^{-1}$. That would create Delta type of functions in the inversion however.

If $T>2T_1$, then $u$ has no singularities in $\overline{B(0,1)}$. Then $\mathbf{u}(T)$ has no singularities conormal to $\pi_\omega$ for every unit $\omega$ because this is true in $\pi_\omega\cap \overline{B(0,1)} $, but also true outside it by the fact that all singularities of $\mathbf{u}(T)$ must be along geodesics issued from $\bar\Omega$; and outside it, $c=1$. Therefore, the missing part of $\mathcal{R}\mathbf{u}(T) (p,\omega) $ for $p>1$ and $\omega$ corresponding to planes intersecting $\overline{B(0,1)}$ is a smoothing operator applied to $f$. We would get a smoothing error if we cut it smoothly to zero for those planes. 

Set
\be{k}
k(p,\omega) := 2c_n^-[\partial_t^{(n+1)/2} Mf](T+1-p,\omega)
\ee
Based on those arguments, choose $\chi\in C^\infty(\R)$ so that $\chi(p)=0$ for $p<1+\eps$ and $\chi(p)=1$ for $p>1+2\eps$. If $0<\eps<(T-2T_1)/4$, then $\chi k$ differs from $\mathcal{R} \mathbf{u}(T) (p,\omega)$ by a smoothing term. Therefore, $\mathcal{R}^{-1}\chi k$ is a parametrix for $\mathbf{u}(T)$. If we use the measurements $N$, then $\chi k$ is defined with $M=N/2$ there, by Theorem~\ref{thm_MN}. 

Next theorem gives a time reversal construction that recovers $f$ up to a smoothing term with full data, when $T>T_1$, i.e., when we have stability (all singularities are visible). 
\begin{thm}
Let  $n\ge3$ be odd, $T>T_1$ and let $\chi$ be as above. Let $v$ be the solution of the acoustic wave equation in $(0,T)\times\R^n$ with Cauchy data $\mathbf{u}(T) = \mathcal{R}^{-1}\chi k$. Then 
\[
f = v|_{t=0}+Rf,
\]
with $R$ a smoothing operator. 
\end{thm}

Since $\supp \mathbf{u}(T)\subset \overline B(0,1+T)$, we can solve the wave equation for $v$ in the cylinder $(0,T)\times B(0,1+T+\eps)$ for some fixed $\eps$ with Dirichlet, Neumann or some kind of absorbing boundary conditions because no singularities of $v$ leave the smaller cylinder corresponding to $\eps=0$.

We have a refined result for partial data when some singularities might not be invisible. 
\begin{thm}\label{thm_TR_p}
Let  $n\ge3$ be odd. 
Let $\chi\in C^\infty_0(\R_+\times\bo)$ and let $k$ be as in \r{k}. Let $T>0$ be such that $\supp \chi \subset [0,T)\times\bo $.  Let $v$ be the solution of the acoustic wave equation in $(0,T)\times\R^n$ with Cauchy data $\mathbf{u}(T) = \mathcal{R}^{-1}\chi k$. Then 
\[
v|_{t=0} = Pf,
\]
where $P$ is a \PDO\ of order zero with a principal symbol
\[
p(x,\xi)= \frac12  \chi\big(t_+(x,\hat\xi) , \dot\gamma_{x,\hat\xi}(t_+(x,\hat\xi))\big) +\frac12\chi\big(t_-(x,\hat\xi) ,-\dot \gamma_{x,\hat\xi}(t_-(x,\hat\xi))\big)   . 
\]
\end{thm}

\begin{proof} 
Consider the mappings
\[
C_0^\infty(\Omega_1)\ni  f \; \mathop{\longrightarrow}^{M}\; Mf \; \mathop{\longrightarrow}^{K}\; k\in C_0^\infty((1,1+T)\times \mathbb{S}^{n-1}),
\]
where $Kh$ is as in \r{k} with $h= Mf$ there, and $\Omega_1$ is a domain such that $\Omega\Subset\Omega_1\Subset B(0,1)$.  The operator $K$ is a composition of a differential operator and a linear transformation of the variables and as such, is a trivial FIO associated with a diffeomorphic canonical relation. Choose $\chi_0$ as $\chi$ in \r{k} but related to $\Omega_1$ now.  
Define $F$ as the operator mapping the  Cauchy data $\mathcal{R}^{-1}\chi_0 k$ at $t=T$   to the solution of the acoustic equation at $t=0$. Then $F$ is a microlocal left parametrix of $KM$ restricted to some conic neighborhood of the singularities visible from $\supp\chi$, i.e., $FKM=\Id$ up to a smoothing operator on that conic neighborhood. 
As we proved above, $M=M_++M_-$, where $M_\pm$ are associated with canonical diffeomorphisms. We will show below that $FKM_\pm=\frac12\Id$ modulo a smoothing operator. On intuitive level, this is clear from the second equation in \r{8}: when each of the singularities of $f$ splits into two, the principal parts of the amplitudes in the geometric optics expansion \r{GOS} of each part at $T=0$ are equal and equal to $1/2$. We compare $FKM_\pm$ (which equals $\frac12\Id$ modulo a \PDO\ of order $-1$)  with $FK\chi M_\pm$. Since $\chi_0\chi=\chi$, 
by the Egorov's theorem (\cite[Theorem~25.3.5]{Hormander4}), $FK\chi M_\pm$ are \PDO s with  principal symbol given by $\chi$ pulled back by the canonical relation of $M_\pm$, which proves the theorem. 

It remains to prove the claim we used in the previous paragraph. It can be easily seen (see \cite{SU2011}), that the $\sigma=\pm$ terms in \r{GOS} that we call $u_\pm$, are parametrices for the wave equation with Cauchy data $\frac12(f, \pm\i \sqrt{-c^2\Delta}f)$ at $t=0$. The operators $M_\pm$ are obtained from them as in \r{plane}. By Theorem~\ref{FIO},  $M_\pm$ have separated ranges (by the sign of $\tau$) and we can use a pseudo-differential partition of unity $A_++A_-$ w.r.t.\ $t$ to separate them, i.e., $A_\pm M=M_\pm$. Then $FKM_\pm f=FKA_\pm M f$. The operator $F$ is just a time reversal of $\mathcal{R}^{-1}\chi_0 k$ from $t=T$ to $t=0$, therefore, $FKM_+ f$ is the first component of $(u_+,\partial_t u_+)$ at $t=0$, which equals $f/2$ modulo a smoothing operator applied to $f$. The same statement follows for  $FKM_+ f$. 
\end{proof}

The theorem allows to construct a parametrix recovering any fixed in advance compact subset of the visible singularities from $(0,T)\times\Gamma$ by choosing $\chi$ equal to one on the image of that subset under $C_+\cup C_-$, and zero near the boundary of $(0,T)\times\Gamma$. Note that $\chi$ could also be a \PDO\ of order zero with obvious modifications of the theorem.

\bibliographystyle{amsalpha}

\end{document}